\documentclass[a4paper,11pt]{amsart}

\usepackage{amsfonts,amssymb,mathtools}
\usepackage{textcomp, color}
\usepackage{hyperref}

\theoremstyle{definition}

\newtheorem{definition}{Definition}[section]

\theoremstyle{remark}

\newtheorem{remark}[definition]{Remark}

\theoremstyle{plain}

\newtheorem{theorem}[definition]{Theorem}
\newtheorem{corollary}[definition]{Corollary}
\newtheorem*{theoremA}{Theorem A}
\newtheorem*{theoremB}{Theorem B}

\newtheorem{lemma}[definition]{Lemma}

\newcommand{\N}{\mathbb{N}}

\newcommand{\maxG}{\max\nolimits_{G}\hspace{.09 em}}

\DeclareMathOperator{\Aut}{Aut}

\DeclareMathOperator{\Cl}{Cl}

\title[Lower central words in finite $p$-groups]{Lower central words in finite $p$-groups}

\author{Iker de las Heras}
\address{Zientzia eta Teknologia Fakultatea, Matematika Saila, Euskal Herriko Unibertsitatea (UPV/EHU), Sarriena Auzoa z/g, 48940 Leioa, Spain.} 
\email{iker.delasheras@ehu.eus}

\author{Marta Morigi}
\address{Dipartimento di Matematica, Universit\`a di Bologna, 
Piazza di Porta San Donato 5, 40126 Bologna, Italy.}
\email{marta.morigi@unibo.it}

\begin{document}

\begin{abstract}
It is well known that the set of values of a lower central word in a group $G$ need not be a subgroup.
For a fixed lower central word $\gamma_r$ and for $p\ge 5$, Guralnick showed that if $G$ is a finite $p$-group such that the verbal subgroup $\gamma_r(G)$ is abelian and 2-generator, then $\gamma_r(G)$ consists only of $\gamma_r$-values.
In this paper we extend this result, showing that the assumption that $\gamma_r(G)$ is abelian can be dropped.
Moreover, we show that the result remains true even if $p=3$.
Finally, we prove that the analogous result for pro-$p$ groups is true.
\end{abstract}

\thanks{The first author is supported by the Spanish Government, grant MTM2017-86802-P, partly with FEDER funds, and by the Basque Government, grant IT974-16. He is also supported by a predoctoral grant of the University of the Basque Country. The second author is a member of INDAM}
\keywords{$p$-groups, commutators, lower central words}
\subjclass[2010]{20D15, 20F12, 20F14}

\maketitle

\section{Introduction}

A word $w$ in $k$ variables is an element of the free group $F_k$ with $k$ generators.
For any group $G$, this word can be seen as a map from the Cartesian product of $k$ copies of $G$ to the group $G$ itself by substituting group elements for the variables.
The image of this map is called the set of $w$-values of $G$ and it is denoted by $G_{w}$.
The subgroup generated by this set is called the verbal subgroup of $w$ in $G$ and is denoted by $w(G)$.

In this paper we will focus on the lower central words.
This words are defined recursively by the rule $\gamma_1(x_1)=x_1$ and
$$
\gamma_r(x_1,\ldots,x_r)=[\gamma_{r-1}(x_1,\ldots,x_{r-1}),x_r]
$$
for $r\ge 2$.
Thus, the verbal subgroup $\gamma_r(G)$ of the word $\gamma_r$ in a group $G$ coincides with the $r$-th term of the lower central series of $G$.
In this context, it is well known that the set of $\gamma_r$-values need not be a subgroup.
In other words, $G_{\gamma_r}$ may be a proper subset of $\gamma_r(G)$.


However, several families of groups have been found for which the equality $\gamma_r(G)=G_{\gamma_r}$ holds.
The study of this property started with the case $r=2$, that is, when the word $\gamma_r$ is the common commutator word and its verbal subgroup is just the derived subgroup of the group.
One of the main results in this case is the proof by Liebeck, O'Brien, Shalev and Tiep in \cite{LOST} of the 
so-called Ore Conjecture, according to which every finite simple group $G$ satisfies the condition $G'=G_{\gamma_2}$.

In the opposite direction, still in the case $r=2$, the result is also true for nilpotent groups with cyclic derived subgroup, as proved by Rodney in \cite{R}.
If, instead, we drop the nilpotency assumption, the result fails to hold.
Namely, in \cite{M}, Macdonald provides some examples of groups
$G$ with $G'$ cyclic and $G'\ne G_{\gamma_2}$.
For finite nilpotent groups, or, equivalently, for finite $p$-groups, Rodney addressed the simplest cases, showing that $G'=G_{\gamma_2}$ if $G'$ is $3$-generator and central or if $G'$ is elementary abelian of rank $3$ (\cite{R2}).
Guralnick extended Rodney's results proving that if $G'$ is abelian, then $G'=G_{\gamma_2}$ whenever $G'$ can be generated by $2$ elements (\cite[Theorem A]{G}) or whenever $G'$ can be generated by $3$ elements and $p\ge 5$ (\cite[Theorem B]{G}).
In addition, Guralnick himself showed that the result is no longer true if $G'$ is $3$-generator and $p=2$ or $p=3$ (\cite{G}, Example 3.5 and Example 3.6).

On this basis, the first author and G.\ A.\ Fern\'andez-Alcober in \cite{DF} and \cite{D} improved Guralnick's results, showing that the condition that $G'$ is abelian can be removed.
Moreover, Macdonald (\cite[Exercise 5, page 78]{M2}) and Kappe and Morse (\cite[Example 5.4]{KM2}) had already shown that for every prime $p$ there exist finite $p$-groups with 4-generator abelian derived subgroup such that $G'\neq G_{\gamma_2}$.
Therefore, for $r=2$, the study of this property for finite $p$-groups in terms of the number of generators of 
the derived subgroup is already completed.

For the case $r>2$, however, much less is known.
The first results were due to Dark and Newell in \cite{DN}, where they generalized Macdonald's and Rodney's results in \cite{M} and \cite{R} to lower central words.
So far, the main results in this context were proved by Guralnick: he showed in \cite{G2} and \cite{G3} that if $G$ is a finite $p$-group, $p\ge 5$, such that $\gamma_r(G)$ is 2-generator and abelian, then $\gamma_r(G)=G_{\gamma_r}$.
In addition, he found an example of a $2$-group such that $\gamma_r(G)\neq G_{\gamma_r}$, but the case $p=3$ remained unknown.

The goal of this paper is to generalize again Guralnick's result, showing that the condition that $\gamma_r(G)$ is abelian is not necessary.
Moreover, we prove that the result is also true if $p=3$, closing in that way the gap between the primes $2$ and $5$.

\begin{theoremA}
Let $G$ be a finite $p$-group and let $r\ge 2$.
If $\gamma_r(G)$ is cyclic or if $p$ is odd and $\gamma_r(G)$ can be generated with $2$ elements, 
then there exist $x_1,\ldots,x_{j-1},x_{j+1},\ldots,x_r\in G$ with $1\le j\le r$ such that
$$
\gamma_r(G)=\{[x_1,\ldots,x_{j-1},g,x_{j+1},\ldots,x_{r}]\mid g\in G\}.
$$
\end{theoremA}

As in \cite{DF} and \cite{D}, we will also prove the analogous version of Theorem A for pro-$p$ groups.
In the case of a pro-$p$-group $G$,
$\gamma_r(G)$ denotes the topological closure of the subgroup generated by the set of all $\gamma_r$-values.

\begin{theoremB}
Let $G$ be a pro-$p$ group and let $r\ge 2$.
If $\gamma_r(G)$ is procyclic or if $p$ is odd and $\gamma_r(G)$ can be topologically generated 
with $2$ elements, 
then there exist $x_1,\ldots,x_{j-1},x_{j+1},\ldots,x_r\in G$ with $1\le j\le r$ such that
$$
\gamma_r(G)=\{[x_1,\ldots,x_{j-1},g,x_{j+1},\ldots,x_{r}]\mid g\in G\}.
$$
\end{theoremB}

\textit{Notation and organization.\/}
Let $G$ be a group.
If $L$ is a normal subgroup of $G$, then $[L,_1\,G]=[L,G]$ denotes the subgroup generated by all commutators $[x,y]$ with $x\in L$ and $y\in G$, and we define recursively $[L,_{n+1}\,G]=[[L,_n\,G],G]$ for all $n\ge 1$.
If $H\le G$ and $x\in G$, then we set $[x,H]=\langle [x,h]\mid h\in H \rangle$.
Moreover, $H^n$ will denote the subgroup generated by all $n$-th powers of elements of $H$.
We denote the Frattini subgroup of $G$ by $\Phi(G)$ and if $G$ is finitely generated, $d(G)$ stands for the minimum number of generators of $G$.
Finally, if $G$ is a topological group, we write $\Cl_G(H)$ to refer to the topological closure of $H$ in $G$ and we write $H\trianglelefteq_{\mathrm{o}} G$ to denote that $H$ is an open normal subgroup of $G$.

We start with some general preliminary results in Section \ref{section preliminaries} that will be used frequently along the paper.
Then we split the proof of Theorem A into three sections, dealing separately with two different cases: first, in Section 
\ref{section cyclic case} we prove the result when $\gamma_r(G)$ is cyclic, and then, in Section \ref{section non-cyclic C=G} and Section \ref{section non-cyclic CneqG} we prove it when $d(\gamma_r(G))=2$ and $p$ is odd, making an additional distinction on the position of a certain subgroup inside the group.
However, the proof for the non-cyclic case in Section \ref{section non-cyclic C=G} and Section \ref{section non-cyclic CneqG} will require further preliminaries that will be developed in Section \ref{section preliminaries non-cyclic}.
Finally, we prove Theorem B in Section \ref{section profinite}.

\section{Preliminaries}
\label{section preliminaries}


Throughout the paper we will use freely the following
well-known commutator identities (see for instance \cite[5.1,5]{Rob}).

\begin{lemma}
\label{comm id}
Let $x,y,z$ be elements of a group.
Then:
\begin{enumerate}
 \item $[x,y]=[y,x]^{-1}$.
  \item $[xy,z]=[x,z]^y[y,z]$, and $[x,yz]=[x,z][x,y]^{z}$.
  \item $\left[x,y^{-1}\right]=[y,x]^{y^{-1}}$, and $\left[x^{-1},y\right]=[y,x]^{x^{-1}}$.
  \item $[x,y^{-1},z]^y\,[y,z^{-1},x]^z\,[z,x^{-1},y]^x=1$ (the Hall-Witt identity).
\end{enumerate}
\end{lemma}

The next standard properties are consequences of the identities above and for the reader convenience we collect them in a lemma that will be often used without mentioning.

\begin{lemma}
\label{lemma prelim}
Let $G$ be a group.
Then:
\begin{enumerate}
 \item If $L$ and $N$ are two normal subgroups of $G$ and $n\in\N$, then $[L^n,N]\le [L,N]^n [L,N,L]$.
  \item If $L$ is a normal subgroup of $G$, then $[L,\gamma_i(G)]\le [L,_iG]$ for every $i\in\N$.
\end{enumerate}
\end{lemma}

We will also use without mentioning the fact that if $N\le L$ are two normal subgroups of $G$ such that 
$[L:N|=p^2$
then $[L,G,G]\le N$, while if $L/N$ is cyclic then $[L,_i\,G]\le L^{p^i}N$ for each $i\in \N$.

The following lemma is essentially the well-known Hall-Petresco Identity (see \cite[Appendix A.1]{berk}).
\begin{lemma}
\label{Hall-Petr}
Let $x,y$ be elements of a group and let $n\in \N$.
Then for each $i=2,\dots, n$ 
there exists $c_i\in \gamma_i(\langle
y,[x,y]\rangle)$ such that 
$$[x,y]^n=[x,y^n]c_2^{\binom n2}c_3^{\binom n3}\cdots c_n^{\binom nn}.$$\end{lemma}


Outer commutator words, also known under the name of multilinear commutator words, are words
obtained by nesting commutators, but using always different variables.
More formally, the word $w(x)=x$ in one variable is an outer commutator word; if $\alpha$ and $\beta$ are outer commutator words involving different variables then the word $w=[\alpha,\beta]$ is an outer commutator, and all outer commutator words are obtained in this way.
Thus, lower central words are particular instances of outer commutator words, and as Lemma \ref{lemma powerful} below shows, the verbal subgroup of such words in finite $p$-groups is powerful whenever it can be generated by $2$ elements.
Hence, the theory of powerful $p$-groups will be essential in this paper.
These groups are usually seen as a generalization of abelian groups since they satisfy, among others, the following properties:
\begin{enumerate}
    \item $\Phi(G)=G^p$. In particular $|G:G^p|=p^{d(G)}$.
    \item $d(H)\le d(G)$ for every $H\le G$.
    \item $G^p=\{g^p\mid g\in G\}$.
    \item If $G=\langle x_1,\ldots,x_n\rangle$, then $G^p=\langle x_1^p,\ldots,x_n^p\rangle$.
    \item The power map from $G^{p^{i-1}}/G^{p^i}$ to $G^{p^i}/G^{p^{i+1}}$ that sends $gG^{p^i}$ to $g^{p}G^{p^{i+1}}$ is an epimorphism for every $i\ge 0$.
\end{enumerate}
A background in such groups can be found, for instance, in \cite[Chapter 2]{DdSMS} or \cite[Chapter 11]{K}.

In order to prove Lemma \ref{lemma powerful} we first need the following result, which is a basic fact about finite $p$-groups.

\begin{lemma}
\label{lemma inclusion}
Let $G$ be a finite $p$-group and $N,K$ normal subgroups of $G$.
If $N\le KN^p[N,G]$, then $N\le K$.
\end{lemma}
\begin{proof}
Factor out $K$ and just note that if $N$ is non-trivial, then $N^p[N,G]$ is a proper subgroup of $N$, which is a contradiction.
\end{proof}

\begin{lemma}
\label{lemma powerful}
Let $G$ be a finite $p$-group and $w$ an outer commutator word.
If $d(w(G))=2$, then $w(G)'\le w(G)^{p^2}$.
In particular $w(G)$ is powerful.
\end{lemma}
\begin{proof}
By Theorem 1 of \cite{B} the result is true if $w$ is the commutator word, so we assume $w(G)\le\gamma_3(G)$.
In order to show that $w(G)'\le w(G)^{p^2}$ we may assume that $w(G)^{p^2}=1$, and by Lemma \ref{lemma inclusion} we can also assume $[w(G)',G]=(w(G)')^p=1$.

Since $d(w(G))=2$ we have $|w(G):\Phi(w(G))|=p^2$, and so $[w(G),G,G]\le\Phi(w(G))$.
Observe first that
$$
[\Phi(w(G)),w(G)]=[w(G)^pw(G)',w(G)]\le(w(G)')^p[w(G)',w(G)]=1,
$$
so in particular $\Phi(w(G))$ is abelian and $\Phi(w(G))^p=(w(G)^p)^p(w(G)')^p=w(G)^{p^2}=1$.
Moreover,
\begin{equation}\label{eq}
    \begin{split}
    [\Phi(w(G)),G]&=[w(G)^pw(G)',G]=[w(G)^p,G][w(G)',G]\\
    &\le[w(G),G]^p[w(G),G,w(G)].
    \end{split}
\end{equation}
We consider now two cases in turn: $[w(G),G]\le\Phi(w(G))$ and $[w(G),G]\not\le\Phi(w(G))$.

If $[w(G),G]\le\Phi(w(G))$, then by (\ref{eq}) we have
\begin{align*}
    [\Phi(w(G)),G]&\le[w(G),G]^p[w(G),G,w(G)]\\
    &\le\Phi(w(G))^p[\Phi(w(G)),w(G)]=1.
\end{align*}
Hence,
\begin{align*}
w(G)'&=[w(G),w(G)]\le [w(G),\gamma_3(G)]\\
&\le[w(G),G,G,G]\le[\Phi(w(G)),G]=1,
\end{align*}
as desired.

Suppose now $[w(G),G]\not\le\Phi(w(G))$.
By (\ref{eq}), we have
\begin{align*}
    [w(G),G,G,G,G]&\le[\Phi(w(G)),G,G]\\
    &\le [[w(G),G]^p[w(G),G,w(G)],G]\\
    &\le [w(G),G,G]^p[w(G),G,G,G,G,G]\\
    &\le\Phi(w(G))^p[w(G),G,G,G,G,G]\\
    &=[w(G),G,G,G,G,G],
\end{align*}
so $[w(G),G,G,G,G]=1$.
In addition, the quotient group
$$
w(G)/[w(G),G]\Phi(w(G))
$$
is cyclic.
Hence,
\begin{align*}
    w(G)'&=[w(G),[w(G),G]\Phi(w(G))]\\
    &\le[w(G),G,G,G,G]=1,
\end{align*}
and the proof is complete.
\end{proof}

Therefore, as we will deal with $2$-generator verbal subgroups, we will always assume that $\gamma_r(G)$ is powerful.
Moreover, the next lemma, proved in Lemma \ref{lemma prelim} of \cite{DF}, shows that actually all the subgroups of 
$\gamma_r(G)$  are also powerful.

\begin{lemma}
\label{lemma powerful subgroups}
Let $G$ be a powerful $p$-group.
If $d(G)=2$, then every subgroup $H$ of $G$ is also powerful.
\end{lemma}

The following result is a particular case of Lemma 3.1 of \cite{D}, where it is proved more generally for potent $p$-groups.

\begin{lemma}
\label{lemma index}
Let $G$ be a powerful $p$-group with $p\ge 3$.
If $N\le L$ are two normal subgroups of $G$, then $|N:N^{p^i}|\le|L:L^{p^i}|$ for all $i\ge 0$.
In particular $|L^{p^i}:N^{p^i}|\le |L:N|$.
\end{lemma}

In order to prove Theorem A we will construct a series of subgroups from $\gamma_r(G)$ to $1$ with the property that every element of each factor group of two consecutive subgroups in the series can be written as a $\gamma_r$-value in a suitable way.
Lemma \ref{lemma linking} below will then allow us to go up in this series, proving that actually all the subgroups in the series consist of $\gamma_r$-values, until we reach $\gamma_r(G)$.
The key part of the proof is the following lemma, which is a generalization to outer commutator words of Lemma 2.1 in \cite{AS}.

\begin{lemma}
\label{lemma separation}
Let $G$ be a group and let $w$ be an outer commutator word in $r$ variables.
Let $y_1,\ldots,y_{j-1},h,y_{j+1},\ldots,y_r\in G$.
Then there exist
$h_1,\ldots,h_r\in\langle h\rangle^G$ such that for every $g\in G$,
\begin{equation*}
\begin{split}
    w(y_1&,\ldots,y_{j-1},gh,y_{j+1},\ldots,y_r)\\
    &=w(y_1^{h_1},\ldots,y_{j-1}^{h_{j-1}},g^{h_j},y_{j+1}^{h_{j+1}},\ldots,y_r^{h_r})w(y_1,\ldots,y_{j-1},h,y_{j+1},\ldots,y_r).
\end{split}
\end{equation*}
\end{lemma}
\begin{proof}
We proceed by induction on the number of variables appearing in the outer commutator word $w$.
If such number is $1$, i.e. if $w=x$, then the result is obvious.
Hence, assume $w=[\alpha,\beta]$, where $\alpha$ and $\beta$ are outer commutator words involving $k$ and $r-k$ variables with $k<r$, respectively.
Assume also that $j> k$, so that

\begin{equation*}
\begin{split}
    w(y_1&,\ldots,y_{j-1},gh,y_{j+1},\ldots,y_r)\\
    &=[\alpha(y_1,\ldots,y_k),\beta(y_{k+1},\ldots,y_{j-1},gh,y_{j+1},\ldots,y_r)].
\end{split}
\end{equation*}
By induction, we have
\begin{equation*}
\begin{split}
    \beta&(y_{k+1},\ldots,y_{j-1},gh,y_{j+1},\ldots,y_r)\\
    &=\beta(y_{k+1}^{h_1},\ldots,y_{j-1}^{h_{j-1}},g^{h_j},y_{j+1}^{h_{j+1}},\ldots,y_r^{h_r})
    \beta(y_{k+1},\ldots,y_{j-1},h,y_{j+1},\ldots,y_r),
\end{split}
\end{equation*}
where $h_{k+1},\ldots,h_r\in\langle h\rangle^G$.

For simplicity, write $z_1=\beta(y_{k+1}^{h_1},\ldots,y_{j-1}^{h_{j-1}},g^{h_j},y_{j+1}^{h_{j+1}},\ldots,y_r^{h_r})$, 
$z_2=\beta(y_{k+1},\ldots,y_{j-1},h,y_{j+1},\ldots,y_r)$, and notice that
\begin{equation*}
\begin{split}
&[\alpha(y_1,\ldots,y_k),z_1z_2]=[\alpha(y_1,\ldots,y_k),z_2][\alpha(y_1,\ldots,y_k),z_1]^{z_2}\\
&=[\alpha(y_1,\ldots,y_k),z_1]^{z_2^{\alpha(y_1,\ldots,y_k)}}[\alpha(y_1,\ldots,y_k),z_2].
\end{split}
\end{equation*}
Since clearly $z_2\in\langle h\rangle^G$, the result follows.

The case $j\le k$ is similar.
\end{proof}

The following result is an easy consequence of Lemma \ref{lemma separation}; it is also proved in \cite[Proposition 1.2.1]{S}.

\begin{corollary}
\label{corollary separation}
Let $G$ be a group.
Then, for every $i=1,\dots,n$ and for every $g,g_1,\dots,g_{i-1},g_{i+1},\dots g_n\in G$, $h\in \gamma_s(G)$ we have
  \begin{equation*}
      \begin{split}
      [g_1,\dots,g_{i-1},gh,g_{i+1},\dots,g_n]&\equiv\\
      [g_1,\dots,g_{i-1},g,g_{i+1},\dots ,g_n&][g_1,\dots,g_{i-1},h,g_{i+1},\dots,g_n]\pmod{\gamma_{n+s}(G)}.
    \end{split}
  \end{equation*}
  In particular, if $h\in G'$ then
  \begin{equation*}
      \begin{split}
      [g_1,\dots,g_{i-1},gh,g_{i+1},\dots,&g_n]\equiv\\
      &[g_1,\dots,g_{i-1},g,g_{i+1},\dots ,g_n]\pmod{\gamma_{n+1}(G)}.
    \end{split}
  \end{equation*}
\end{corollary}

\begin{lemma}
\label{lemma linking}
Let $G$ be a group and $w$ an outer commutator word on $r$ variables.
Let $N\le L\le G$ with $N$ normal in $G$ and suppose that for some $x_1,\ldots,x_{j-1},x_{j+1},\ldots,x_r\in G$, the following two conditions hold:
\begin{enumerate}
    \item $L\subseteq \bigcup_{g\in G}Nw(y_1,\ldots,y_{j-1},g,y_{j+1},\ldots,y_r)\text{ for every }y_i\in x_i^G.$
    \item $N\subseteq\{w(y_1,\ldots,y_{j-1},g,y_{j+1},\ldots,y_r)\mid g\in G\}$ for every $y_i\in x_i^G$.
\end{enumerate}
Then, $L\subseteq\{w(y_1,\ldots,y_{j-1},g,y_{j+1},\ldots,y_r)\mid g\in G\}$ for every $y_i\in x_i^G$.
\end{lemma}
\begin{proof}
Take an arbitrary coset $Nw(y_1,\ldots,y_{j-1},h,y_{j+1},\ldots,y_r)$ of $N$ in $L$, with $y_i\in x_i^G$ and $h\in G$.
Take $h_1,\ldots,h_r$ as in Lemma \ref{lemma separation} and let 
$z$ be an arbitrary element of $N$.
By assumption, there exists $u\in G$ such that $z=w(y_1^{h_1},\ldots,y_{j-1}^{h_{j-1}},u,y_{j+1}^{h_{j+1}},\ldots,y_r^{h_r})$
and we may also assume that $u$ is of the form $u=g^{h_j}$ with $g\in G$.

So, by Lemma \ref{lemma separation} our arbitrary element $zw(y_1,\ldots,y_{j-1},h,y_{j+1},\ldots,y_r)$ 
of the above coset can be written as
\begin{equation*}
\begin{split}
    w(y_1^{h_1},\ldots,y_{j-1}^{h_{j-1}},&g^{h_j},y_{j+1}^{h_{j+1}},\ldots,y_r^{h_r})w(y_1,\ldots,y_{j-1},h,y_{j+1},\ldots,y_r)\\
    &=w(y_1,\ldots,y_{j-1},gh,y_{j+1},\ldots,y_r),
\end{split}
\end{equation*}
as desired.
\end{proof}

We end this section with the following three technical lemmas, which will be basically used to
introduce powers inside commutators in the factor groups of the series of $\gamma_r(G)$ mentioned before
Lemma \ref{lemma separation}.
In particular, Lemma \ref{lemma petrescohard}
will be especially useful to prove that these factor groups consists only of some suitable $\gamma_r$-values.

\begin{lemma}
\label{lemma petrescoeasy}
Let $G$ be a finite $p$-group such that for some $r\ge 2$ we have $d(\gamma_r(G))\le2$ if $p$ is odd or $d(\gamma_r(G))=1$ if $p=2$.
Then,
$$
[x_1,\ldots,x_r]^{p^k}\equiv [[x_1,\ldots,x_j]^{p^k},x_{j+1},\ldots,x_r]\pmod{\gamma_r(G)^{p^{k+1}}}
$$
for every $x_1,\ldots,x_r\in G$, $k\ge 0$ and $2\le j\le r$.
Moreover, if $[x_1,\ldots,x_i]\in R$ for some normal subgroup $R$ of $G$ and $1\le i\le j$, then
$$
[x_1,\ldots,x_r]^{p^k}\equiv [[x_1,\ldots,x_j]^{p^k},x_{j+1},\ldots,x_r]\pmod{[R,_{r-i}\,G]^{p^{k+1}}}.
$$
\end{lemma}
\begin{proof}
The first assertion follows immediately from the second one.
We fix $r$, and we will prove by induction on $r-j$ that the assertion holds for all $k$.
Thus, assume $[x_1,\ldots,x_i]\in R$ for some normal subgroup $R$ of $G$ and some $1\le i\le j$.
For $r=j$ the result is clear, so assume $j<r$ and
$$
[x_1,\ldots,x_r]^{p^k}\equiv [[x_1,\ldots,x_{j+1}]^{p^k},x_{j+2},\ldots,x_r]\pmod{[R,_{r-i}\, G]^{p^{k+1}}}.
$$
By the Hall-Petresco Identity, we have
$$
[x_1,\ldots,x_{j+1}]^{p^k}=[[x_1,\ldots,x_j]^{p^{k}},x_{j+1}]c_2^{\binom{p^k}{2}}\ldots c_{p^k}
$$
with $c_n\in\gamma_n(\langle[x_1,\ldots,x_{j+1}],[x_1,\ldots,x_j]\rangle)$ for $2\le n\le p^k$.
Since $j\ge 2$, it follows that
$$
c_n\in[R,_{nj-i+1}\,G]\le [R,_{j-i+2(n-1)+1}\,G]\qquad 
$$
for every $n$.
Note that $\binom{p^k}{n}\ge p^k-(n-2)$ if $p$ is odd and $\binom{p^k}{n}\ge p^k-(n-1)$ if $p=2$.
We denote with $\lceil s\rceil$ the smallest integer which is greater or equal to $s$.
So, if $p$ is odd, we get
$$
c_n^{\binom{p^k}{n}}\in[R,_{j-i+2(n-1)+1}\,G]^{\lceil p^{k-(n-2)}\rceil},
$$
and if $p=2$ we get
$$
c_n^{\binom{2^k}{n}}\in[R,_{j-i+2(n-1)+1}\,G]^{\lceil 2^{k-(n-1)}\rceil}.
$$
Since $d(\gamma_r(G))\le2$, it follows by Lemma \ref{lemma powerful} that $\gamma_r(G)$ is powerful.
By Lemma \ref{lemma powerful subgroups} we then obtain that for all $m\ge 0$, $[R,_{j-i}\,G]^{p^m}$
is also poweful and $d([R,_{j-i}\,G]^{p^m})\le 2$,
so 
$$
|[R,_{j-i}\,G]^{p^m}:[R,_{j-i}\,G]^{p^{m+1}}|\le p^2
$$
for all $m\ge 0$.
This implies, in particular, that
$$
[[R,_{j-i}\,G]^{p^m},G,G]\le[R,_{j-i}\,G]^{p^{m+1}},
$$
for all $m\ge 0$, and therefore
$$
[R,_{j-i+2(n-1)+1}\,G]^{\lceil p^{k-(n-2)}\rceil}\le[R,_{j-i+1}\,G]^{p^{k+1}}.
$$
Now, if $p$ is odd, using the inductive hypothesis with $k+1$ in place of $k$  we have
\begin{align*}
    [[R,_{j-i+2(n-1)+1}\,G]^{\lceil p^{k-(n-2)}\rceil},_{r-j-1}\,G]&\le[[R,_{j-i+1}\,G]^{p^{k+1}}],_{r-j-1}\,G]\\
    &\le[R,_{r-i}\,G]^{p^{k+1}}.
\end{align*}

If $p=2$ the result follows arguing in the same way, taking into account the fact that, in this case, $\gamma_r(G)$ is cyclic and hence
$$
[[R,_{r-i}\,G]^{2^m},G]\le[R,_{r-i}\,G]^{2^{m+1}}.
$$
\end{proof}

\begin{lemma}
\label{lemma petrescoout}
Let $G$ be a finite $p$-group such that for some $r\ge 2$ we have $d(\gamma_r(G))\le 2$ if $p$ is odd and $d(\gamma_r(G))=1$ if $p=2$.
Assume that $H$ and $K$ are normal subgroups of $G$, with $K$ generated by 
$\gamma_{j-1}$-values.
Then for every $k\ge 0$ and for every $j$ with $1\le j\le r$, we have
$$
[K,H^{p^k},_{r-j}\,G]\le[K,H,_{r-j}\,G]^{p^k}.
$$
\end{lemma}
\begin{proof}
We use induction on $k$.
The case $k=0$  is trivial, so assume $k=1$ first, and suppose $p\ge 3$ (if $p=2$ the proof follows in the same way).
As $p$ divides $\binom p i$ for $2\le i <p$ and $\gamma_3(\langle [K,H],H\rangle)\le [K,H,H,H]$, the Hall-Petresco Identity yields
$$
[K,H^p]\le[K,H]^p[K,H,H,H].
$$
Note that $[K,H]$ is generated by elements of the type $[x_1,\dots,x_{j-1},x_j]^p$,
where 
$x_1,\dots,x_{j-1}\in G$ and $x_j\in H$, so
by Lemma \ref{lemma petrescoeasy}, we have
$$
[[K,H]^p,_{r-j}\,G]\le[K,H,_{r-j}\,G]^p.
$$
On the other hand, $\gamma_r(G)$ is powerful by Lemma \ref{lemma powerful}.
Thus, it follows from 
Lemma \ref{lemma powerful subgroups} that
$$
|[K,H,_{r-j}\,G]:[K,H,_{r-j}\,G]^p|\le p^2,
$$
so we get
\begin{align*}
    [K,H,H,H,_{r-j}\,G]&\le[[K,H,_{r-j}\,G],G,G]\\
    &\le[K,H,_{r-j}\,G]^p.
\end{align*}
Hence,
$$
[K,H^p,_{r-j}\,G]\le
[[K,H]^p[K,H,H,H]],_{r-j}\,G]
\le [K,H,_{r-j}\,G]^p,
$$
as desired.

Assume now $k\ge 2$.
Then, by induction,
\begin{align*}
    [K,H^{p^k},_{r-j}\,G]&\le[K,(H^p)^{p^{k-1}},_{r-j}\,G]\\
    &\le [K,H^p,_{r-j}\,G]^{p^{k-1}}\\
    &\le ([K,H,_{r-j}\,G]^p)^{p^{k-1}},
\end{align*}
and since $[K,H,_{r-j}\,G]$ is powerful by Lemma \ref{lemma powerful subgroups}, we have
$$
([K,H,_{r-j}\,G]^p)^{p^{k-1}}=[K,H,_{r-j}\,G]^{p^k}.
$$
\end{proof}

\begin{lemma}
\label{lemma petrescohard}
Let $G$ be a finite $p$-group and let $N,L$ be normal subgroups of $G$ such that $\gamma_r(G)^p\le N\le L \le \gamma_r(G)$ with $r\ge 2$ and $|L:N|=p$.
Assume that there exist some $j$ with $1\le j\le r$ and 
$x_1,\ldots,x_{j-1},h,x_{j+1},\ldots,x_r\in G$ such that
\begin{align*}
    L&=\langle[x_1,\ldots,x_{j-1},h,x_{j+1},\ldots,x_r]\rangle N.
\end{align*}
Let $H$ be the normal closure of $\langle h\rangle$ in $G$ and assume also that one of the following conditions hold:
\begin{enumerate}
    \item $p$ is odd, $d(\gamma_r(G))\le2$ and the subgroup
    $$
    [\gamma_{j}(G),H,H,_{r-j}\,G]
    $$
    is central of exponent $p$ modulo $N^p$.
    
    \item $p=2$, the subgroup $\gamma_r(G)$ is cyclic and
    \begin{align*}
    [x_1,\ldots,x_{j-1}&,h,x_{j+1},\ldots,x_r]^{2}\\
    &\equiv[x_1,\ldots,x_{j-1},h^{2},x_{j+1},\ldots,x_r]\pmod{N^{2}}.
    \end{align*}
\end{enumerate}
Then,
\begin{align*}
    [x_1,\ldots,x_{j-1}&,h,x_{j+1},\ldots,x_r]^{p^k}\\
    &\equiv[x_1,\ldots,x_{j-1},h^{p^k},x_{j+1},\ldots,x_r]\pmod{N^{p^k}}
\end{align*}
for every $k\ge 0$.
In particular,
$$
L^{p^k}=\langle[x_1,\ldots,x_{j-1},h^{p^k},x_{j+1},\ldots,x_r]\rangle N^{p^k}.
$$
\end{lemma}
\begin{proof}
We use induction on $k$.
If $k=0$ there is nothing to prove and, if $p=2$ and $k=1$, then the result follows from the hypothesis.
Thus, assume $k\ge 1$ if $p$ is odd or $k\ge 2$ if $p=2$, and suppose, by induction, that
$$
[x_1,\ldots,x_{j-1},h,x_{j+1},\ldots,x_r]^{p^{k-1}}=[x_1,\ldots,x_{j-1},h^{p^{k-1}},x_{j+1},\ldots,x_r]y
$$
for some $y\in N^{p^{k-1}}$.

Let $u=[x_1,\ldots,x_{j-1},h^{p^{k-1}},x_{j+1},\ldots,x_r]\in \gamma_r(G)$.
Note that $(uy)^p=
u^py^pc$ where $c\in [N^{p^{k-1}},\gamma_r(G)]\le [N^{p^{k-1}},G,G]\le (N^{p^{k-1}})^p=N^{p^{k}}$.
Thus,
\begin{equation*}
\begin{split}
    ([x_1,\ldots,x_{j-1}&,h^{p^{k-1}},x_{j+1},\ldots,x_r]y)^p\\
    &\equiv[x_1,\ldots,x_{j-1},h^{p^{k-1}},x_{j+1},\ldots,x_r]^p\pmod{N^{p^k}}.
\end{split}
\end{equation*}
Moreover, by Lemma \ref{lemma petrescoout}, we have
\begin{align*}
[\gamma_{j}(G),H^{p^{k-1}},_{r-j}\,G]^{p^2}&\le[\gamma_{j}(G),H,_{r-j}\,G]^{p^{k+1}}\\
&\le(\gamma_r(G)^{p^{k+1}})\le N^{p^k},
\end{align*}
so using Lemma \ref{lemma petrescoeasy} with $R=[\gamma_{j-1}(G),H^{p^{k-1}}]$ we obtain
\begin{align*}
    [x_1,\ldots,x_{j-1},&h,x_{j+1},\ldots,x_r]^{p^{k}}\\
    &\equiv[x_1,\ldots,x_{j-1},h^{p^{k-1}},x_{j+1},\ldots,x_r]^p\\
    &\equiv[[x_1,\ldots,x_{j-1},h^{p^{k-1}}]^p,x_{j+1},\ldots,x_r]\pmod{N^{p^k}}.
\end{align*}

Suppose now $p$ is odd.
We first prove that
\begin{equation}\label{star}
 [\gamma_{j-1}(G),H^{p^{k-1}},H^{p^{k-1}},_{r-j}\,G]
\text{ is central of exponent }p \text{ modulo }N^{p^k}.\end{equation}

If $k=1$ the claim follows from the hypothesis, so we may assume $k\ge 2$.
Recall that $L$, $N$ and $[\gamma_{j-1}(G),H,H,_{r-j}\,G]$ are powerful by Lemma \ref{lemma powerful} and Lemma \ref{lemma powerful subgroups}.
From Lemma \ref{lemma petrescoout} we then get
\begin{align*}
    [\gamma_{j-1}(G),H^{p^{k-1}}&,H^{p^{k-1}},_{r-j+1}\,G]\\
    &\le [\gamma_{j-1}(G),H,H,_{r-j+1}\,G]^{p^{2k-2}}\\
    &\le (N^p)^{p^{2k-2}}\le N^{p^k}
\end{align*}
and
\begin{align*}
    [\gamma_{j-1}(G),H^{p^{k-1}}&,H^{p^{k-1}},_{r-j}\,G]^p\\
    &\le ([\gamma_{j-1}(G),H,H,_{r-j}\,G]^{p^{2k-2}})^p\\
    &\le (N^p)^{p^{2k-2}}\le N^{p^k}.
\end{align*}
This proves (\ref{star}).

By the Hall-Petresco Identity, since $p\ge 3$, we get 
$$
[x_1,\ldots,x_{j-1},h^{p^{k-1}}]^p=[x_1,\ldots,x_{j-1},h^{p^k}]z_2^pz_3,
$$
where $z_i\in\gamma_i(\langle[x_1,\ldots,x_{j-1},H^{p^{k-1}}],H^{p^{k-1}}\rangle)$ for $i=2,3$.
Write
$$
R=[\gamma_{j-1}(G),H^{p^{k-1}},H^{p^{k-1}}],
$$
so that $z_2\in R$ and $z_3\in[R,G].$

On the one hand, by (\ref{star}) we have
\begin{align*}
[z_3,x_{j+1},\ldots,x_r]\in [R,_{r-j+1}\,G]\le N^{p^k}.
\end{align*}

On the other hand it follows from Lemma \ref{lemma petrescoout} 
with $H=R$ and $K=G$ and from (\ref{star})
that
\begin{align*}
    [z_2^p,x_{j+1},\ldots,x_{r}]\in[R,_{r-j}\,G]^p\le N^{p^k}.
\end{align*}
Therefore,
\begin{align*}
    [x_1,\ldots,x_{j-1}&,h,x_{j+1},\ldots,x_r]^{p^k}\equiv [[x_1,\ldots,x_{j-1},h^{p^k}]z_2^pz_3,x_{j+1},\ldots,x_r]\\
    &\equiv[x_1,\ldots,x_{j-1},h^{p^k},x_{j+1},\ldots,x_r]\pmod{N^{p^k}}
\end{align*}
as we wanted.

If $p=2$, since $\gamma_r(G)$ is cyclic, we have $L=\gamma_r(G)$, $N=\gamma_r(G)^{p}$ and 
the inductive step easily follows 
from the Hall-Petresco Identity.
Namely, 
$$
[x_1,\ldots,x_{j-1},h^{2^{k-1}}]^2=[x_1,\ldots,x_{j-1},h^{2^k}]z_2,
$$
where $z_2\in [\gamma_{j-1}(G),G^{2^{k-1}},G^{2^{k-1}}]$.
By Lemma \ref{lemma petrescoout} we have
$$[\gamma_{j-1}(G),G^{2^{k-1}},G^{2^{k-1}},_{r-j}\,G]\le \gamma_{r+1}(G)^{2k-2}\le \gamma_{r}(G)^{2^{k+1}}=N^{2^k},$$
so the result follows as above.\end{proof}

\section{Proof of Theorem A when $\gamma_r(G)$ is cyclic}
\label{section cyclic case}

Dark and Newell already proved Theorem A when $\gamma_r(G)$ is cyclic in \cite{DN}, but we will give an alternative simpler proof in Theorem \ref{theorem cyclic p=2} below.
In addition, we will also prove the case $p=2$, which was omitted since it was pointed out to be very technical.
Moreover, even if Theorem \ref{theorem cyclic p=2} can be modified so that it works for all primes, we will prove the case in which $p$ is odd separately in Theorem \ref{theorem p odd},
since in this case the proof turns out to be much shorter.
First, however, we need the following simple but very helpful lemma.

\begin{lemma}
\label{lemma aut}
Let $N$ be a cyclic normal subgroup of a group $G$.
Then, $[N,G']=1$.
\end{lemma}
\begin{proof}
Since $N$ is cyclic, the automorphism group $\Aut(N)$ of $N$ is abelian.
Hence, $G/C_G(N)$ is also abelian, which means that $G'\le C_G(N)$.
\end{proof}

We will also need the following result, which is Lemma 2.3 of \cite{DF}.

\begin{lemma}\label{2.3DF}
Let $G$ be a group and let $N\le L \le G$, with $N$ normal in $G$.
Suppose that for some $x\in G$ the following two conditions hold:
\begin{itemize}
 \item[(i)] $L/N \subseteq \left\{N[x,g]\,| g\in G\right\}$.
 \item[(ii)] $N\subseteq \left\{[x,g]\,| g\in G\right\}$.
\end{itemize}
Then $L\subseteq\left\{[x,g]\,| g\in G\right\}$.
\end{lemma}

\begin{theorem}
\label{theorem p odd}
Let $G$ be a finite $p$-group with $p$ odd and $\gamma_r(G)$ cyclic.
Then
$$\gamma_r(G)=\{[g_1,\ldots,g_r]\mid g_1,\ldots,g_r\in G\}.$$
\end{theorem}
\begin{proof}
Let $\gamma_r(G)=\langle[x_1,\ldots,x_r]\rangle$ with $x_1,\ldots,x_r\in G$.
Then,
$$
\gamma_r(G)^{p^k}=\langle[x_1,\ldots,x_r]^{p^k}\rangle
$$
for every $k\ge 1$.
By the Hall-Petresco Identity, we have
$$
[x_1,\ldots,x_r]^{p^k}=[x_1,\ldots,x_r^{p^k}]c_2^{\binom{p^k}{2}}\ldots c_{p^k}
$$
with $c_i\in\gamma_i(\langle[x_1,\ldots,x_r],x_r\rangle)$.
When $i<p^k$, we have $c_i\in\gamma_{r+i-1}(G)\le\gamma_r(G)^{p^{i-1}}$, and so $c_i^{\binom{p^k}{i}}\in \gamma_r(G)^{p^{k+1}}$ since $p\ge 3$.
If $i=p^k$, then $c_{p^k}\in\gamma_{r+p^k-1}(G)\le \gamma_r(G)^{p^{p^k-1}}\le\gamma_r(G)^{p^{k+1}}$.
Therefore,
$$
\gamma_r(G)^{p^k}=\langle[x_1,\ldots,x_r^{p^k}]\rangle
$$
for every $k\ge 0$.
Moreover, since $[x_1,\ldots,x_r^{p^k},G]\le\gamma_r(G)^{p^{k+1}}$, we have
$$
[x_1,\ldots,x_r^{p^k}]^i\equiv[x_1,\ldots,x_r^{ip^k}]\pmod{\gamma_r(G)^{p^{k+1}}}
$$
for every $i\ge0$, so the result follows from Lemma \ref{2.3DF}.
\end{proof}

\begin{theorem}
\label{theorem cyclic p=2}
Let $G$ be a finite $2$-group with $\gamma_r(G)$ cyclic.
Then
$$\gamma_r(G)=\{[g_1,\ldots,g_r]\mid g_1,\ldots,g_r\in G\}.$$
\end{theorem}
\begin{proof} Define $C=C_G(\gamma_r(G)/\gamma_r(G)^{4})$.
Since $\gamma_r(G)$ is cyclic, the quotient group $\gamma_r(G)/\gamma_r(G)^{4}$ has order $4$, so that $|G:C|\le 2$.
Let $\gamma_r(G)=\langle[x_1,\ldots,x_r]\rangle$ with $x_1,\ldots,x_r\in G$ and let $j$ be the maximum 
number such that $x_j\in C$.
Assume, in addition, that $[x_1,\ldots,x_r]$ is, among all $\gamma_r$-values which are generators of $\gamma_r(G)$,
the one with maximum $j$ (observe that $j\ge 2$ since $G'=[G,C]$).

For every $i=1,\ldots,r$ consider an arbitrary element $y_i\in x_i^G$, so that $y_i=x_i[x_i,g]$ for some $g\in G$.
Since $\gamma_{r+1}(G)\le\gamma_r(G)^2$, it follows from Corollary \ref{corollary separation} that
$$[y_1,\ldots,y_r]\equiv [x_1,\ldots,x_r]\pmod{\gamma_r(G)^2},$$
and since $\gamma_r(G)^2=\Phi(\gamma_r(G))$, we have
$$
\gamma_r(G)=\langle[y_1,\ldots,y_r]\rangle.
$$
Therefore
$$
\gamma_{r}(G)^{2^k}=\langle[y_1,\ldots,y_r]^{2^k}\rangle
$$
for every $k\ge 1$.
We claim that
$$
[y_1,\ldots,y_r]^{2^k}\equiv[y_1,\ldots,y_j^{2^k},\ldots,y_r]\pmod{\gamma_r(G)^{2^{k+1}}}
$$
for every $y_i\in x_i^G$ and $k\ge 1$.
Take $k=1$ first.
By lemma \ref{lemma petrescoeasy} we have
$$
[y_1,\ldots,y_r]^{2}\equiv[[y_1,\ldots,y_j]^{2},y_{j+1},\ldots,y_r]\pmod{\gamma_r(G)^4},
$$
and observe that
$$
[y_1,\ldots,y_j^{2},\ldots,y_r]=[[y_1,\ldots,y_j]^{2}[y_1,\ldots,y_j,y_j],y_{j+1},\ldots,y_r].
$$
If 
$$
[y_1,\ldots,y_j,y_j,y_{j+1},\ldots,y_r]\not\in \gamma_r(G)^4,
$$
then
$$
\gamma_{r+1}(G)=\gamma_r(G)^2=\langle[y_1,\ldots,y_j,y_j,y_{j+1},\ldots,y_r]\rangle,
$$
and so
$$
\gamma_{r}(G)=\langle[y_1,\ldots,y_j,y_j,y_{j+1},\ldots,y_{r-1}]\rangle,
$$
which contradicts the maximality of $j$ in the choice of the generator $[x_1,\ldots,x_r]$.

Hence,
$$
[y_1,\ldots,y_j,y_j,y_{j+1},\ldots,y_r]\in \gamma_r(G)^4,
$$
so that
$$
[y_1,\ldots,y_r]^{2}\equiv[y_1,\ldots,y_j^{2},\ldots,y_r]\pmod{\gamma_r(G)^{4}}.
$$
The claim follows now from Lemma \ref{lemma petrescohard} with $L=\gamma_r(G)$, $N=\gamma_r(G)^p$.

Now we can conclude our proof.
Let $2^m$ be the order of $\gamma_r(G)$.
We will prove by 
induction on $m-k$ that
$$
\gamma_r(G)^{2^k}\subseteq \{[g_1,\ldots,g_r]\mid g_1,\ldots,g_r\in G\}.
$$
The result is true when $k=m$, so assume $k<m$ and
$$
\gamma_r(G)^{2^{k+1}}\subseteq \{[g_1,\ldots,g_r]\mid g_1,\ldots,g_r\in G\}.
$$
We apply Lemma \ref{lemma linking} with $L=\gamma_r(G)^{2^{k-1}}$ and
$N=\gamma_r(G)^{2^{k}}$.
As
$$
L=[y_1,\ldots,y_j^{2^k},\ldots,y_r]N\cup N\subseteq 
\bigcup_{g\in G}\gamma_r(y_1,\ldots,y_{j-1},g,y_{j+1},\ldots,y_r)N
$$\vspace{-8pt}

\noindent for every $y_i\in x_i^G$, by Lemma \ref{lemma linking} 
we get
$$
\gamma_r(G)^{2^{k}}\subseteq \{[g_1,\ldots,g_r]\mid g_1,\ldots,g_r\in G\}.
$$
In particular, when $k=0$ we obtain
$$
\gamma_r(G)\subseteq \{[g_1,\ldots,g_r]\mid g_1,\ldots,g_r\in G\},
$$
as we wanted.
\end{proof}

Thus, combining Theorem \ref{theorem p odd} and Theorem \ref{theorem cyclic p=2} we get the result for all primes when $\gamma_r(G)$ is cyclic.

\section{Preliminaries for the proof of Theorem A when $\gamma_r(G)$ is generated by $2$ elements}
\label{section preliminaries non-cyclic}

We will use the following notation: if $H,K$ are subgroups of a group $G$, by $U\max\nolimits_{H}\hspace{.09 em}K$
we mean that $U$ is maximal among the proper subgroups of $K$ which are normalized by $H$, while $U\max K$ simply 
means that $U$ is a maximal subgroup of $K$.

The subgroups defined in Definition \ref{definition D} and Definition \ref{definition E} will be essential in our proof.

\begin{definition}
\label{definition D}
Let $G$ be a finite $p$-group and let $U\maxG \gamma_r(G)$ for some $r\ge 2$.
We define
$$
D_r(U)=C_{\gamma_{r-1}(G)}(G/U).
$$
In other words, for $x\in\gamma_{r-1}(G)$ we have $x\in D_r(U)$ if and only if $[x,G]\le U$.
\end{definition}

\begin{definition}
\label{definition E}
Let $G$ be a finite $p$-group and let $U\max_{\gamma_{r-1}(G)} \gamma_r(G)$ for some $r\ge 2$.
We define
$$
E_r(U)=C_G(\gamma_{r-1}(G)/U).
$$
In other words, $x\in E_r(U)$ if and only if $[x,\gamma_{r-1}(G)]\le U$.
\end{definition}

\begin{remark}
The subset $E(U)$ may not be a subgroup of $G$ if $U$ is not normal in $G$.
\end{remark}

The significance of these subgroups becomes clear in the following Lemma.

\begin{lemma}
\label{lemma DE}
Let $G$ be a finite $p$-group and let $r\ge 2$.
Then, for $x\in\gamma_{r-1}(G)$, we have $\gamma_r(G)=[x,G]$ if and only if
$$
x\not\in \bigcup\{D_r(U)\mid U\maxG \gamma_r(G)\}.
$$
Similarly, $\gamma_r(G)=[\gamma_{r-1}(G),y]$ if and only if
$$
y\not\in\bigcup\{E_r(U)\mid U\,\text{\emph{max}}_{\gamma_{r-1}(G)}\gamma_r(G)\}.
$$
\end{lemma}
\begin{proof}
The proof is essentially the same as the one of Lemma 2.9 of \cite{DF}.
Let $x \in\gamma_{r-1}(G)$.
Since $[x, G]$ is a normal subgroup of $G$, we have $[x, G] < \gamma_r(G)$ if and
only if $x\in D_r(U)$ for some $U\maxG \gamma_r(G)$, and the first assertion follows.
Similarly, since $[\gamma_{r-1}(G), y]$ is normalized by $\gamma_{r-1}(G)$, we have $[\gamma_{r-1}(G), y] < \gamma_r(G)$ if and
only if $y\in E_r(U)$ for some $U\max\nolimits_{\gamma_{r-1}(G)}\hspace{.09 em}\gamma_r(G)$.
\end{proof}

\begin{lemma}
\label{lemma DE2}
Let $G$ be a finite $p$-group with $d(\gamma_r(G))=2$ for some $r\ge 2$.
Let $U,V,W\maxG\gamma_{r}(G)$ with $V\neq W$ and $R,S,T\max\nolimits_{\gamma_{r-1}(G)}\gamma_r(G)$ with $S\neq T$.
Then,
\begin{enumerate}
    \item $D_r(U)\neq \gamma_{r-1}(G)$ and $E_r(R)\neq G$.
    \item $D_r(V)\cap D_r(W)\le D_r(U)$ and $E_r(S)\cap E_r(T)\subseteq E_r(R)$.
    \item If $U\neq R$, then $[D_r(U),E_r(R)]\le\gamma_r(G)^p$.
\end{enumerate}
\end{lemma}
\begin{proof}
(i) is obvious, since $D_r(U)=\gamma_{r-1}(G)$ implies that $\gamma_r(G)\le U$ and similarly $E_r(R)=G$ implies that $\gamma_r(G)\le R$, and in both cases we have a contradiction.
   
We now prove (ii).
As $d(\gamma_r(G))=2$, the subgroup $\gamma_r(G)$ is powerful by Lemma \ref{lemma powerful}, so $\gamma_r(G)^p=\Phi(\gamma_r(G))$.
Hence, $V\cap W\le \gamma_r(G)^p\le U$ and $S\cap T\le \gamma_r(G)^p\le R$.
Then, the result follows from the fact that $x\in D_r(V)\cap D_r(W)$ if and only if $[x,G]\le V\cap W$ and $y\in E_r(S)\cap E_r(T)$ if and only if $[y,\gamma_{r-1}(G)]\le S\cap T$.
     
(iii) is true because $[D_r(U),E_r(R)]\le U\cap R\le\gamma_r(G)^p$.
\end{proof}

The following subgroup plays a fundamental role in \cite{G}, \cite{DF} and \cite{D}, and so does in our proof.

\begin{definition}
Let $G$ be a finite $p$-group.
We define
$$
C_r(G)=C_G(\gamma_r(G)/\gamma_r(G)^p).
$$
\end{definition}

\begin{lemma}
\label{lemma C}
Let $G$ be a finite $p$-group with $d(\gamma_r(G))=2$ for some $r\ge 2$.
Then:
\begin{enumerate}
    \item $|G:C_r(G)|\le p.$
    
    \item We have $G=C_r(G)$ if and only if $\gamma_{r+1}(G)\le \gamma_r(G)^p$.
    In this case, all subgroups $U$  such that $\gamma_r(G)^p< U<\gamma_r(G)$ are normal in $G$.
    Otherwise, $C_r(G)\neq G$ and there is only one normal subgroup $U$ of $G$  such that $\gamma_r(G)^p<U<\gamma_r(G)$, namely $U=\gamma_{r+1}(G)\gamma_r(G)^p$.
    
    \item We have $[\gamma_r(G)^{p^k},C_r(G)]\le\gamma_r(G)^{p^{k+1}}$ for all $k\ge 0$.
\end{enumerate}
\end{lemma}
\begin{proof}
By Lemma \ref{lemma powerful} the subgroup $\gamma_r(G)$ is powerful, so $\gamma_r(G)/\gamma_r(G)^p$ is an elementary abelian $p$-group of rank $2$.
Now (i) follows from the fact that the quotient group $G/C_r(G)$ embeds in a Sylow $p$-subgroup of the automorphism group of $\gamma_r(G)/\gamma_r(G)^p$.

To prove (ii), we may assume that $\gamma_r(G)^p=1$.
There are precisely $p+1$ non-trivial proper subgroups of $\gamma_r(G)$, all cyclic of order $p$, and each of them is normal in $G$ if and only if it is central.
In addition, all such subgroups are central if and only $G=C_r(G)$, which is equivalent to $\gamma_{r+1}(G)=1$.
If there exists a non central subgroup $U$ of $G$ with $1\ne U< \gamma_r(G)$ then the conjugacy class of $U$ has size $p$, $C_r(G)\neq G$ and $\gamma_{r+1}(G)\ne 1$ is the only non-trivial normal subgroup of $G$ properly contained in $\gamma_r(G)$.
This proves (ii).

The proof of (iii) is an easy induction on $k$.
The base of the induction is given by the definition of $C_r(G)$, and if $k>0$
then
\begin{align*}
    [\gamma_r(G)^{p^k},C_r(G)]&\le [\gamma_r(G)^{p^{k-1}},C_r(G)]^p[\gamma_r(G)^{p^{k-1}},C_r(G),\gamma_r(G)^{p^{k-1}}] \\
    &\le\gamma_r(G)^{p^{k+1}}[\gamma_r(G)^{p^{k}},\gamma_r(G)]\le \gamma_r(G)^{p^{k+1}}
\end{align*}
by using the inductive hypothesis and the fact that $\gamma_r(G)$ is powerful.
\end{proof}

In the case $r=2$, i.e. when we deal with the common commutator word, we will also need the next lemma, which is just Lemma 2.9 (i) of \cite{DF}.

\begin{lemma}
\label{lemma C2}
If $G$ is a non-abelian finite $p$-group with $d(G')\le 2$, then for every $U\maxG G'$, we have $D_2(U)\le C_2(G)$.
\end{lemma}

\section{Proof of Theorem A when $C_r(G)=G$}
\label{section non-cyclic C=G}

In order to apply Lemma \ref{lemma petrescohard} we will first find in Lemma \ref{lemma generation C=G} suitable generators for the verbal subgroup $\gamma_r(G)$.
Then, as mentioned before, we will conclude by applying Lemma \ref{lemma linking}.

\begin{lemma}
\label{lemma generation C=G}
Let $G$ be a finite $p$-group with $d(\gamma_r(G))=2$ for some $r\ge 2$.
If $C_r(G)=G$, then there exist an integer $j$ with $1\le j\le r$ and $x_1,\ldots,x_{j-1},x_{j+1},\ldots,x_r\in G$ such that
$$
\gamma_r(G)=\langle[y_1,\ldots,y_{j-1},g,y_{j+1},\ldots,y_r]\mid g\in G\rangle
$$
for every $y_i\in x_i^G$.
\end{lemma}
\begin{proof} We may assume that $\Phi(\gamma_{r}(G))=1$, so using Lemma \ref{lemma C} (ii) we also have $\gamma_{r+1}(G)\le\gamma_r(G)^p=1$.
Notice that it suffices to find an integer $j$ and $x_1,\ldots,x_{j-1},x_{j+1},\ldots,x_r\in G$ such that
$$
\gamma_r(G)=\langle[x_1,\ldots,x_{j-1},g,x_{j+1},\ldots,x_r]\mid g\in G\rangle,
$$
since if $y_i\in x_i^G$, then $y_i=x_ih_i$ for some $h_i\in G'$, so it follows from Corollary \ref{corollary separation} that $[y_1,\ldots,y_{j-1},g,y_{j+1},\ldots,y_r]=[x_1,\ldots,x_{j-1},g,x_{j+1},\ldots,x_r]$.
 
We will proceed by induction on $r$.
If $r=2$, then the result is true by the aforementioned Theorem A of \cite{DF}.

Now, if there exists $x\in G_{\gamma_{r-1}}$ such that $\gamma_r(G)=[x,G]$ then we are done.
Hence, suppose $[x,G]<\gamma_r(G)$ for every $x\in G_{\gamma_{r-1}}$.
Observe that all subgroups $U$ such that $\gamma_r(G)^p\le U\le\gamma_r(G)$ are normal in $G$ by Lemma \ref{lemma C} (ii), so we have
$$
U\maxG\gamma_r(G)\ \ \text{ for every }\ \ U\max\gamma_r(G).
$$

If
$$
D=\prod_{V\max\gamma_r(G)} D_r(V)<\gamma_{r-1}(G),
$$
then we could choose a $\gamma_{r-1}$-value not belonging to $D$, which contradicts Lemma \ref{lemma DE}.
Therefore, assume
$$
\prod_{V\max\gamma_r(G)} D_r(V)=\gamma_{r-1}(G).
$$
Thus, by (i) and (ii) of Lemma \ref{lemma DE2}, there exists $U\max\gamma_r(G)$ such that $D_r(U)$ properly contains $\bigcap\{D_r(V)\mid V\max\gamma_r(G)\}$, and therefore, $[D_r(U),G]=U$.
Now, by Lemma \ref{lemma DE2} (iii), we have $[D_r(U),E_r(V)]=1$ for all $V\neq U$, and so
$$
\prod_{\substack{V\max\gamma_{r}(G)\\U\neq V}}E_r(V)\neq G.
$$
Hence, as $G$ can not be the union of two proper subgroups, we can choose
$$
x_r\in G\setminus \Big(E_r(U)\bigcup\prod_{\substack{V\max\gamma_{r}(G)\\U\neq V}}E_r(V)\Big),
$$
and observe that by Lemma \ref{lemma DE} we have
$$
\gamma_r(G)=[\gamma_{r-1}(G),x_r].
$$
Define now $C_{x_r}=C_{\gamma_{r-1}(G)}(x_r)$ and notice that $C_{x_r}$ is normal in $G$ since 
$$
[C_{x_r},G,x_r]\le [\gamma_{r-1}(G),G,x_r]\le\gamma_{r+1}(G)=1.
$$
Thus, we consider the quotient group $G/C_{x_r}$.
Since $\gamma_{r+1}(G)=1$ the map
\begin{equation*}
  \begin{split}
         \eta: \gamma_{r-1}(G)&\longrightarrow \gamma_{r}(G)\\
        g\ \ \ \ &\longmapsto [g,x_r]
    \end{split}
\end{equation*}
is a group epimorphism whose kernel is $C_{x_r}$, so 
$$
|\gamma_{r-1}(G/C_{x_r})|=p^2.
$$
Furthermore, since $\gamma_{r+1}(G)=1$, we have $C_{r-1}(G/C_{x_r})=G/C_{x_r}$.
By inductive hypothesis, there exist an integer $j$ with $1\le j\le r-1$ and $x_1,\ldots,x_{j-1},x_{j+1},\ldots,x_{r-1}\in G$ such that
$$
\gamma_{r-1}(G)=\langle[x_1,\ldots,x_{j-1},g,x_{j+1},\ldots,x_{r-1}]\mid g\in G\rangle C_{x_r}.
$$

Finally,
\begin{align*}
    \gamma_r(G)&=[\gamma_{r-1}(G),x_r]\\
    &=[\langle[x_1,\ldots,x_{j-1},g,x_{j+1},\ldots,x_{r-1}]\mid g\in G\rangle C_{x_r},x_r]\\
    &=\langle[x_1,\ldots,x_{j-1},g,x_{j+1},\ldots,x_{r-1},x_r]\mid g\in G\rangle,
\end{align*}
and this concludes the proof.
\end{proof}

\begin{theorem}
\label{theorem C=G}
Let $G$ be a finite $p$-group with $p$ odd and $d(\gamma_r(G))=2$.
If $C_r(G)=G$, then there exist an integer $j$ with $1\le j\le r$ and $x_1,\ldots,x_{j-1},$ $x_{j+1},\ldots,x_r\in G$ such that
$$
\gamma_r(G)=\{[x_1,\ldots,x_{j-1},g,x_{j+1},\ldots,x_r]\mid g\in G\}.
$$
\end{theorem}
\begin{proof}
By Lemma \ref{lemma generation C=G}, there exist exist an integer
$j$ with $1\le j\le r$ and $x_1,\ldots,x_{j-1},x_{j+1},\ldots,x_r\in G$ such that
$$
\gamma_r(G)=\langle[y_1,\ldots,y_{j-1},g,y_{j+1},\ldots,y_r]\mid g\in G\rangle
$$
for every $y_i\in x_i^G$.
Choose arbitrarily $y_i\in x_i^G$ for all $i$.
We have
\begin{equation*}
    \begin{split}
    \gamma_r(G)=&\langle[y_1,\ldots,y_{j-1},g_1,y_{j+1},\ldots,y_r],\\
    &\ \ \ [y_1,\ldots,y_{j-1},g_2,y_{j+1},\ldots,y_r]\rangle
\end{split}
\end{equation*}
for some $g_1,g_2\in G$.
Let
$$
U=\langle[y_1,\ldots,y_{j-1},g_2,y_{j+1},\ldots,y_r]\rangle\gamma_r(G)^p,
$$
and notice that it is normal in $G$ since $C_r(G)=G$.
Observe that $\gamma_{r+1}(G)\le\gamma_r(G)^p$, and $\gamma_r(G)^p$ is central of exponent $p$ modulo $\gamma_{r}(G)^{p^2}$ by 
(iii) of Lemma \ref{lemma C}.
Therefore, we apply Lemma \ref{lemma petrescohard} to both quotients
$$
\gamma_r(G)/U\ \ \text{ and }\ \ U/\gamma_r(G)^p
$$
and we get 
$$
\gamma_r(G)^{p^k}=\langle[y_1,\ldots,y_{j-1},g_1^{p^k},y_{j+1},\ldots,y_r]\rangle U^{p^k}
$$
and
$$
U^{p^k}=\langle[y_1,\ldots,y_{j-1},g_2^{p^k},y_{j+1},\ldots,y_r]\rangle \gamma_r(G)^{p^{k+1}}
$$
for every $k\ge 0$.
Furthermore, as $\gamma_{r+1}(G)\le \gamma_r(G)^p$, it follows from Corollary \ref{corollary separation} that 
$$
[y_1,\ldots,y_{j-1},g_1,y_{j+1},\ldots,y_r]^s\equiv [y_1,\ldots,y_{j-1},g_1^s,y_{j+1},\ldots,y_r]\pmod{U}
$$
and
$$
[y_1,\ldots,y_{j-1},g_2,y_{j+1},\ldots,y_r]^s\equiv [y_1,\ldots,y_{j-1},g_2^s,y_{j+1},\ldots,y_r]\hspace{-6pt}\pmod{\gamma_r(G)^p}
$$
for each integer $s$.
Thus, using Lemma \ref{lemma petrescohard} and the aforementioned property (v) of powerful $p$-groups it can be easily proved that
\begin{align*}
    [y_1,\ldots,y_{j-1},g_1^{p^k}&,y_{j+1},\ldots,y_r]^s\\
    &\equiv [y_1,\ldots,y_{j-1},g_1,y_{j+1},\ldots,y_r]^{sp^k}\\
    &\equiv ([y_1,\ldots,y_{j-1},g_1^s,y_{j+1},\ldots,y_r]u)^{p^k}\\
    &\equiv [y_1,\ldots,y_{j-1},g_1^{sp^k},y_{j+1},\ldots,y_r]\pmod{U^{p^k}},
\end{align*}
where $u\in U$, and similarly 
\begin{align*}
    [y_1,\ldots,y_{j-1},g_2^{p^k}&,y_{j+1},\ldots,y_r]^s\\
    &\equiv [y_1,\ldots,y_{j-1},g_2^{sp^k},y_{j+1},\ldots,y_r]\pmod{\gamma_r(G)^{p^{k+1}}}.
\end{align*}
Hence, for each $k\ge 0$ we have
$$
\gamma_r(G)^{p^k}\subseteq \bigcup_{g\in G}\gamma_r(y_1,\ldots,y_{j-1},g,y_{j+1},\ldots,y_r)U^{p^k}
$$
for every $y_i\in x_i^G,
$ and similarly $$U^{p^k}\subseteq\bigcup_{g\in G}\gamma_r(y_1,\ldots,y_{j-1},g,y_{j+1},\ldots,y_r)\gamma_r(G)^{p^{k+1}}$$ for every $y_i\in x_i^G$.

The result now follows by repeatedly applying Lemma \ref{lemma linking} to the subgroups of the series
$$1=\gamma_r(G)^{p^s} \le U^{p^{s-1}}\le \gamma_r(G)^{p^{s-1}}\le \dots\le \gamma_r(G)^{p^{i}}\le U^{p^{i-1}}\le\dots\le
\gamma_r(G),$$
wkere $p^s$ is the exponent of $\gamma_r(G)$.
\end{proof}

\section{Proof of Theorem A when $C_r(G)\neq G$}
\label{section non-cyclic CneqG}

To end the proof of Theorem A, we need a further technical definition.

\begin{definition}
Let $G$ be a finite $p$-group and let $r\ge 2$.
We define $C_r^r(G)=\gamma_r(G)^p$ and
$$
C_i^r(G)=C_{\gamma_i(G)}(G/C^r_{i+1}(G))
$$
for all $2\le i \le r-1$.
\end{definition}

As done in Section \ref{section non-cyclic C=G}, we start finding suitable generators for $\gamma_r(G)$.



\begin{lemma}
\label{lemma generation C neq G}
Let $G$ be a finite $p$-group with $d(\gamma_r(G))=2$ for some $r\ge 2$ and $C_r(G)\neq G$.
Let $U=\gamma_{r+1}(G)\gamma_r(G)^p$.
Then, there exist an integer $j$ with $2\le j\le r$, $x_1,\ldots,x_{j-1}\in G$ and $c\in C_r(G)$  such that
$$
\gamma_r(G)=\langle[y_1,\ldots,y_{j-1},c,g_{j+1},\ldots,g_r]\rangle U
$$
for every $y_k\in x_k^G$ with $k=1,\ldots,j-1$ and every $g_{j+1},\ldots,g_{r}\in G\setminus C_r(G)$.
Moreover, $[\gamma_{i}(G),C_r(G)]\le C^r_i(G)$ for every $j\le i\le r$.
\end{lemma}
\begin{proof}  
We proceed by induction on $r$.
Suppose first $r=2$ and take $x\in G\setminus C_2(G)$ arbitrary.
Since $C_2(G)$ is maximal in $G$ by Lemma \ref{lemma C} (i), we have $G=\langle x\rangle C_2(G)$.
Also, as $D_2(U)\le C_2(G)$ by Lemma \ref{lemma C2}, we have $x\not\in D_2(U)$.
Moreover, by Lemma \ref{lemma C} (ii), $U$ is the unique subgroup such that $U\maxG \gamma_r(G)$, so by Lemma \ref{lemma DE} we have $G=[x,G']$.
Thus we get
$$
G'=[x,G]=[x,\langle x\rangle C_2(G)]=[x,C_2(G)].
$$
In addition, $[G',C_2(G)]\le(G')^p=C^2_2(G)$, as desired.

Take then $r\ge 3$ and write $C=C_r(G)$ for simplicity.
We may assume $\gamma_r(G)^p=C^r_r(G)=1$.
Suppose first there exists $x_1,\ldots,x_{r-1}\in G$ such that $\gamma_r(G)=[x_1,\ldots,x_{r-1},C]$.
Since $[\gamma_r(G),C]=1$ and since $x_i^{g}=x_i[x_i,g]$ for evrey $g\in G$, it follows from Corollary \ref{corollary separation} that
$$
\gamma_r(G)=[y_1,\ldots,y_{r-1},C]
$$
for all $y_i\in x_i^G$.
Hence, we may assume there is no such an element.
In other words, if $x\in G_{\gamma_{r-1}}$, then $[x,C]\neq \gamma_{r}(G)$.
Note, however, that $[x,C]$ is normal in $G$ since, as above, $[x,C]^g=[x^g,C]=[x,C]$.
Since $U$ is the only non-trivial normal subgroup of $G$ properly contained in $\gamma_r(G)$, we get $[x,C]\le U$ for every $\gamma_{r-1}$-value $x$.
Since $\gamma_{r-1}(G)$ is generated by all $\gamma_{r-1}$-values, we have, then, $[\gamma_{r-1}(G),C]\le U$.
This, in particular, implies that $C\le E_r(U)$, and since $E_r(U)\neq G$ by Lemma \ref{lemma DE2}, we have $C=E_r(U)$.
Note that we have $V\max_{\gamma_{r-1}(G)}\gamma_r(G)$ for every $V\max\gamma_r(G)$ since
$$
[\gamma_r(G),\gamma_{r-1}(G)]\le[\gamma_r(G),G']\le[\gamma_r(G),G,G]=1.
$$
On the other hand,  $U=\gamma_{r+1}(G)$, so for every 
$V\max\gamma_r(G)$ with $V\neq U$ we have $[\gamma_r(G),E_r(V)]\le U\cap V=1$, and then, $E_r(V)\le C$.
Therefore,
$$
\bigcup \{E_r(V)\mid V\max\gamma_r(G)\}\subseteq C
$$
and then, by Lemma \ref{lemma DE}, we get
$$
\gamma_r(G)=[\gamma_{r-1}(G),g]
$$
for every $g\in G\setminus C$.

As $[\gamma_r(G),\gamma_{r-1}(G)]=1$, the map
\begin{equation*}
    \begin{split}
        \eta_g:\gamma_{r-1}(G)&\longrightarrow \gamma_{r}(G)\\
        x\ \ \ \ &\longmapsto [x,g]
    \end{split}
\end{equation*}
is a group epimorphism for every $g\in G\setminus C$ whose kernel is $C_{\gamma_{r-1}(G)}(g)$.
Choose an arbitrary $g\in G\setminus C$, write $C_g=C_{\gamma_{r-1}(G)}(g)$ for simplicity and note that 
$$
[C_g,G]=[C_g,\langle g\rangle C]=[C_g,C]\le [\gamma_{r-1}(G),C]\le U \le C_g,
$$
where the last equality holds since $U\le Z(G)$.
Thus, the subgroups $C_g$ are all normal in $G$, and we can consider the groups $G/C_g$.
Now, $\gamma_{r-1}(G/C_g)=\gamma_{r-1}(G)/C_g$ is isomorphic to $\gamma_{r}(G)$, so it has order $p^2$ and exponent $p$.
In addition $\gamma_r(G)\not\le C_g$ since otherwise $[\gamma_r(G),g]=1$, which contradicts the fact that $g\not\in C$.
Thus,
$$
G/C_g\neq C_{r-1}(G/C_g).
$$
Moreover, since $[\gamma_{r-1}(G),C]\le U\le C_g$, it follows that
$$
C_{r-1}(G/C_g)=C/C_g
$$
for all $g\in G\setminus C$.
By Lemma \ref{lemma C} (ii), there is only one normal subgroup $R$ of $G$  with $C_g< R<\gamma_{r-1}(G)$, so $R=C_g\gamma_r(G)$.

We apply now the inductive hypothesis to all groups $G/C_g$.
It follows that for each $g\in G\setminus C$, there exist $j_g\ge 1$,
$x_{1,g},\ldots,x_{j_g-1,g}\in G$ and $c_g\in C$ such that
$$
\gamma_{r-1}(G)=\langle[y_{1,g},\ldots,y_{j_g-1,g},c_g,g_{j_g+1},\ldots,g_{r-1}]\rangle C_g\gamma_r(G)
$$
for every $y_{i,g}\in x_{i,g}^G$, $i=1,\ldots,j_g-1$ and every $g_{j_g+1},\ldots,g_{r-1}\in G\setminus C$.
Moreover, if we define
$$C_{i,g}/C_g=C^{r-1}_i(G/C_g),$$
then we have $[\gamma_{i}(G),C]\le C_{i,g}$ for all $j_g\le i\le r-1$.

Define now
$$
U^*=\gamma_r(G)\prod_{g\in G\setminus C}C_g,
$$
which is, of course, normal in $G$.

We claim that $U^*=C_g\gamma_r(G)$ for all $g\in G\setminus C$.
For that purpose, fix $g\in G\setminus C$ and take $h\in G\setminus C$ arbitrary.
Then $C_gC_h$ is normal in $G$, so either $C_gC_h=\gamma_{r-1}(G)$ or $C_h\le C_g\gamma_r(G)$.
In the first case we would have
$$
\gamma_r(G)=[\gamma_{r-1}(G),h]=[C_hC_g,h]=[C_g,h]\le C_g,
$$
which is a contradiction since $[\gamma_r(G),g]\ne 1$.
Hence, $C_h\le C_g\gamma_r(G)$, and so $C_g\gamma_r(G)=C_hC_g\gamma_r(G)$.
Since this holds for all $h\in G\setminus C$, it follows that $C_g\gamma_r(G)=U^*$, and the claim is proved.

Take now $j=\max\{j_g\mid g\in G\setminus C\}$.
Then, there exist
$x_{1},\ldots,x_{j-1}\in G$ and $c\in C$ such that
$$
\gamma_{r-1}(G)=\langle[y_{1},\ldots,y_{j-1},c,g_{j+1},\ldots,g_{r-1}]\rangle U^*
$$
for every $y_i\in x_i^G$, $i=1,\ldots,j-1$ and every $g_{j+1},\ldots,g_{r-1}\in G\setminus C$.
Moreover, because of the choice of $j$, we have
$$
[\gamma_{i}(G),C]\le \bigcap_{g\in G\setminus C}C_{i,g}
$$
for all $j\le i\le r-1$.
Let us prove that $$\bigcap_{g\in G\setminus C}C_{i,g}\le C^r_{i}(G)\text{ for every }i\text{ such that }j\le i\le r-1.$$

We proceed by induction on $r-i$.
If $r-i=1$, that is, if $i=r-1$, then $C_{r-1,g}=C_g=C_{\gamma_{r-1}}(g)$, and since $G=\langle G\setminus C\rangle$, it follows that
$$
\bigcap_{g\in G\setminus C}C_g=C_{\gamma_{r-1}(G)}(G)=C^r_{r-1}(G).
$$
Assume now $i\le r-2$.
Then,
$$
\Big[\hspace{-2pt}\bigcap_{g\in G\setminus C}\hspace{-4pt}C_{i,g}\,,G\Big]\le \bigcap_{g\in G\setminus C}C_{i+1,g}\le C^r_{i+1}(G),
$$
 by the inductive hypothesis, and so,
$$
\bigcap_{g\in G\setminus C}C_{i,g}\le C^r_{i}(G)
$$
as claimed.

Since $[\gamma_r(G),C]=1=C^r_r(G)$, we have $[\gamma_i(G),C]\le C^r_i(G)$ for every $i$ such that $j\le i\le r$.

Finally, take $g_r\in G\setminus C$ arbitrary.
Observe that 
$$
[U^*,g_r]=[C_{g_r}\gamma_r(G),g_r]=[\gamma_r(G),g_r]=U,
$$
where the last equality holds since $1\neq[\gamma_r(G),g_r]\le \gamma_{r+1}(G)$.
Hence,
\begin{align*}
    \gamma_{r}(G)&=[\gamma_{r-1}(G),g_r]\\
    &=[\langle[y_{1},\ldots,y_{j-1},c,g_{j+1},\ldots,g_{r-1}]\rangle U^*,g_r]\\
    &=[\langle[y_{1},\ldots,y_{j-1},c,g_{j+1},\ldots,g_{r-1}]\rangle,g_r]U\\
    &=\langle[y_{1},\ldots,y_{j-1},c,g_{j+1},\ldots,g_r]\rangle U,
\end{align*}
and the proof is complete.
\end{proof}

\begin{theorem}
Let $G$ be a finite $p$-group with $p$ odd and $d(\gamma_r(G))=2$ for some $r\ge 2$.
If $C_r(G)\neq G$, then there exist an integer $j$ with $1\le j\le r$ and $x_1,\ldots,x_{j-1},x_{j+1},\ldots,x_r$ 
such that
$$
\gamma_r(G)=\{[x_1,\ldots,x_{j-1},c,x_{j+1},\ldots,x_r]\mid c\in C_r(G)\}.
$$
\end{theorem}
\begin{proof}
Let $U=\gamma_{r+1}(G)\gamma_r(G)^p$ and write $C=C_r(G)$ 
for simplicity.
By Lemma \ref{lemma generation C neq G}, there exist an integer $j$ with $1\le j\le r$ and $x_1,\ldots,x_{j-1}\in G$,
 $c\in C$ such that
$$
\gamma_r(G)=\langle[y_1,\ldots,y_{j-1},c,g_{j+1},\ldots,g_r]\rangle U
$$
for every $y_i\in x_i^G$, $i=1,\ldots,j-1$ and every $g_{j+1},\ldots,g_{r}\in G\setminus C$.
Moreover, $[\gamma_{i}(G),C]\le C^r_i(G)$ for every $j\le i\le r$.

Write $x=[y_1,\ldots,y_{j-1}]$.
It follows from the Hall-Witt Identity and standard commutator calculus that
$$
[x,c,g_{j+1}]=[c,g_{j+1},x]^{-1}[g_{j+1},x,c]^{-1}z
$$
for some $z\in\gamma_{j+2}(G)$.
On the one hand, we have
$$
[z,g_{j+2},\ldots,g_r]\in \gamma_{r+1}(G)\le U.
$$
On the other hand,
$$
[g_{j+1},x,c]\in[\gamma_j(G),C]\le C^r_j(G)\cap\gamma_{j+1}(G),
$$
and since $[C^r_i(G),G]\le C^r_{i+1}(G)$ for every $i\le r-1$, we have
\begin{align*}
   [C^r_j(G)\cap\gamma_{j+1}(G),g_{j+2},\ldots,g_{r}]\le C^r_{r-1}(G)\cap\gamma_{r}(G)\le U,
\end{align*}
where the last inequality holds since $C^r_{r-1}(G)\cap\gamma_r(G)$ is normal in $G$ but $\gamma_r(G)\not\le C^r_{r-1}(G)$.
Thus,
$$
[x,c,g_{j+1},\ldots,g_r]\equiv[x,[c,g_{j+1}],g_{j+2},\ldots,g_r]\pmod{U},
$$
so in particular
$$
\gamma_r(G)=\langle [x,[c,g_{j+1}],g_{j+2},\ldots,g_r]\rangle U.
$$
Take now $g_{r+1}\in G\setminus C$ arbitrary.
Since, clearly, we have $[U,g_{r+1}]\le\gamma_r(G)^p$, it follows that
$$
U=\langle [x,[c,g_{j+1}],g_{j+2},\ldots,g_{r+1}]\rangle\gamma_r(G)^p.
$$

Now, observe that on the one hand we have
\begin{align*}
    [\gamma_{j-1}(G),C,C,_{r-j}\,G]&\le[\gamma_j(G),C,_{r-j}\,G]\\
    &\le [C^r_{j}(G),_{r-j}\,G]\\
    &\le C^r_r(G)=\gamma_r(G)^p,
\end{align*}
which is central of exponent $p$ modulo $U^{p}$, and on the other hand we have
$$
[\gamma_{j-1}(G),G',G',_{r-j}\,G]\le\gamma_{r+3}(G)\le U^p,
$$
which is central of exponent $p$ modulo $\gamma_r(G)^{p^2}$.
Therefore, we can apply Lemma \ref{lemma petrescohard} to both quotients
$$
\gamma_r(G)/U\ \ \text{ and }\ \ U/\gamma_r(G)^p
$$
and we conclude in the same way as in the proof of Theorem \ref{theorem C=G}.
\end{proof}

\section{Proof of Theorem B}
\label{section profinite}

Now we prove Theorem B using a similar idea as in Theorem B of \cite{DF} and Theorem A$'$ and Theorem B$'$ of \cite{D}.

\begin{proof}[Proof of Theorem B]
We first claim that there exists $1\le j\le r$ such that for every $N\trianglelefteq_{\mathrm{o}} G$ there exist $g_{N,1},\ldots,g_{N,j-1},g_{N,j+1},\ldots,g_{N,r}\in G$ such that
$$
\gamma_r(G)N/N=\{[g_{N,1},\ldots,g_{N,j-1},g,g_{N,j+1},\ldots,g_{N,r}]\mid g\in G\}.
$$
For every $N\trianglelefteq_{\mathrm{o}} G$, write $j_N$ for the smallest integer such that there exist $g_{N,1},\ldots,g_{N,j_N-1},g_{N,j_N+1},\ldots,g_{N,r}\in G$ 
such that
$$
\gamma_r(G)N/N=\{[g_{N,1},\ldots,g_{N,j_N-1},g,g_{N,j_N+1},\ldots,g_{N,r}]\mid g\in G\}.
$$
Note that the existence of $j_N$ is guaranteed by Theorem A.

Let $M$ be an open normal subgroup of $G$ for which $j_M$ is maximal in the set $\{j_N\mid N\trianglelefteq_{\mathrm{o}} G\}$.
We will prove that $j=j_M$ has the required property.
Indeed, take $N\trianglelefteq_{\mathrm{o}} G$ arbitrary and consider the intersection $N\cap M$, which is also open and normal in $G$.
Now, as $N\cap M\le M$, we have $j_M\le j_{N\cap M}$, and by maximality, it follows that $j_M=j_{N\cap M}$.
Again, since $N\cap M\le N$, we have
$$
\gamma_r(G)N/N=\{[g_{N,1},\ldots,g_{N,j_M-1},g,g_{N,j_M+1},\ldots,g_{N,r}]\mid g\in G\},
$$
and the claim is proved.

Now, for every $N\trianglelefteq_{\mathrm{o}} G$, write
\begin{align*}
X_N=\big\{(g_1,\ldots,g_{j-1}&,g_{j+1},\ldots,g_r)\in G\times\overset{r-1}{\ldots}\times G\mid\\ &\hspace{-4pt}\gamma_r(G)N/N=\{[g_1,\ldots,g_{j-1},g,g_{j+1},\ldots,g_r]N\mid g\in G\} \big\}.
\end{align*}
Clearly, the family $\{X_N\}_{N\trianglelefteq_{\mathrm{o}}G}$ has the finite intersection property, and since $G\times\overset{r-1}{\ldots}\times G$ is compact,
$$
\bigcap_{N\trianglelefteq_{\mathrm{o}} G} \, X_N \ne \varnothing.
$$
Thus, if $(g_1,\ldots,g_{j-1},g_{j+1},g_r)$ belongs to this intersection, write
$$
\mathcal{K}(G)=\{[g_1,\ldots,g_{j-1},g,g_{j+1},\ldots,g_r]\mid g\in G\},
$$
so that we have
$$
\gamma_r(G)N/N=\mathcal{K}(G)N/N
$$
for all $N\trianglelefteq_{\mathrm{o}} G.$

Now, note that $\mathcal{K}(G)$ is closed in $G$, being the image of a continuous function from $G$ to $G$.
Thus,


$$\gamma_r(G)=\bigcap_{N\trianglelefteq_{\mathrm{o}}G}\gamma_r(G)N=\bigcap_{N\trianglelefteq_{\mathrm{o}}G}\mathcal{K}(G)N=\Cl_{G}(\mathcal{K}(G))=\mathcal{K}(G)$$
and the proof is complete.
\end{proof}

\end{document}